\documentclass[10pt]{amsart}
\usepackage[utf8]{inputenc}
\usepackage[T1]{fontenc}
\usepackage[frenchb, english]{babel}
\usepackage{amsmath, amssymb, amsthm}
\title{Rapid Stabilization in a Semigroup Framework}
\author{Ambroise Vest}
\address{Institut de Recherche Math\'ematique Avanc\'ee,  Universit\'e de Strasbourg\\7 rue Ren\'e Descartes, 67084 Strasbourg C\'edex, France}
\email{ambroise.vest@math.unistra.fr}
\date{January 24, 2013}
\newcommand{\C}{\mathbb{C}}

\newcommand{\R}{\mathbb{R}}
\newcommand{\N}{\mathbb{N}}

\newcommand{\om}{\omega}
\newcommand{\Lom}{\Lambda_{\omega}}
\newcommand{\ILom}{\Lambda_{\omega}^{-1}}

\newcommand{\Tom}{T_{\omega}}
\newcommand{\eom}{e_{\omega}}
\newcommand{\la}{\langle}

\newcommand{\ra}{\rangle}
\newcommand{\wJ}{\widetilde{J}}
\newcommand{\wA}{\widetilde{A}}

\newtheorem{thm}{Theorem}[section]
\newtheorem{lem}[thm]{Lemma}
\newtheorem{prop}[thm]{Proposition}

\theoremstyle{definition}
\newtheorem*{deff}{Definition}

\theoremstyle{remark}
\newtheorem*{rk}{Remark}

\theoremstyle{remark}

\theoremstyle{remark}

\theoremstyle{remark}

\begin{document}
\maketitle
\begin{abstract}
We prove the well-posedness of a linear closed-loop system with an explicit (already known) feedback leading to arbitrarily large decay rates. We define a mild solution of the closed-loop problem using a dual  equation and  we prove that the original operator perturbed by the feedback is (up to the use of an extension) the infinitesimal generator of a strongly continuous group. We also give a justification to the exponential decay of the solutions. Our method is direct and avoids the use of optimal control theory.
\end{abstract}
\section{Introduction}\label{Introduction}
We consider a physical system which state $x$ satisfies the Cauchy problem 
\begin{equation*} 
\begin{cases}
x'(t)=Ax(t)+Bu(t), \\
x(0)=x_0,
\end{cases}
\end{equation*}
where $A$ is a linear differential operator that models the dynamics of the system and $B$ is a  control operator that allows us to act on the system through a control \nolinebreak $u$.

The \emph{stabilization problem} consists in finding a \emph{feedback operator} $F$ such that the solutions of the \emph{closed-loop problem}
\begin{equation*}
x'=(A+BF)x
\end{equation*}
tend to zero as $t$ tends to $+\infty$. 

For finite-dimensional systems, D. L. Lukes \cite{Lukes1968} and D. L. Kleinman \cite{Kleinman1970} (see also the book of D. L. Russell \cite[pp. 112-117]{Russell1979}) 
gave a systematic stabilization method thanks to an explicit feedback constructed with the  controllability Gramian  
\begin{equation*}
\Lambda:=\int_0^T e^{-tA}BB^*e^{-tA^*}dt.
\end{equation*}
The above matrix is positive-definite provided that  $(A,B)$ is exactly controllable (equivalently $(-A^*,B^*)$ is observable). In this case, the feedback
\begin{equation*}
F:=-B^*\Lambda^{-1}
\end{equation*}
stabilizes the system.

Later, adding a suitable weight-function inside the Gramian operator, M. Slemrod \cite{Slemrod}  adapted and improved this result to the case of infinite-dimensional systems with  bounded control operators. More precisely, his feedback depends on a tuning parameter $\omega >0$ that ensures a prescribed exponential decay rate of the solutions. The weighted Gramian 
\begin{equation*}
\Lom:=\int_0^T e^{-2\omega t} e^{-tA}BB^*e^{-tA^*}dt
\end{equation*}
is positive-definite if $(A,B)$ is exactly controllable in time $T$ (or $(-A^*,B^*)$ is exactly observable in time $T$). Then the solutions of the closed-loop problem provided with the feedback
\begin{equation*}
F:=-B^*\ILom
\end{equation*}
decrease to zero with an exponential decay rate being at least $\omega$ i.e. there is a positive constant $c$ such that
\begin{equation*}
\|x(t)\|\le ce^{-\om t}\|x_0\|, \qquad t\ge 0,
\end{equation*}
for all initial data $x_0$, where $\|\cdot\|$ denotes a norm on the state space.

The problems that we have in mind are linear time-reversible partial differential equations (waves, plates...) with \emph{boundary control}. These  are infinite-dimensional problems and controlling only at (a part of) the boundary of the domain imposes that the control operator $B$ is unbounded. This leads to difficulties in choosing the right functional spaces and the right notion of solution to have well-posed open-loop and closed-loop problems. 

J.-L. Lions \cite{Lions1988} gave an answer to the stabilization of such systems. His proof, using the theory of optimal control, is non-constructive and does not give any information on the decay rate of the solutions. By using a slightly different  weight function in the above  operator $\Lom$, V. Komornik \cite{K97}  gave an explicit feedback leading to arbitrarily large decay rates. His approach does not use the theory of optimal control : an advantage is that one does not have to use strong existence and uniqueness results for infinite dimensional Riccati equations. In fact, the weight function is chosen in such a way that $\Lom$ is the solution of an algebraic Riccati equation. Formally,
\begin{equation*}
A\Lom+\Lom A^*+\Lom C^* C\Lom -BB^*=0,
\end{equation*}
where a definition of the operator $C$ and the rigourous meaning of this equation will be given later.
For a presentation of this method of stabilization, see also the books of V. Komornik and P. Loreti \cite[pp. 23-31]{KL05} (where a generalization of this method to partial stabilization is also given) and J.-M. Coron \cite[pp. 347-351]{Coron07}.

Applications of this method to the boundary stabilization of the wave equation and the plates equation are given in \cite{K97}. This method can also be used to stabilize Maxwell equations \cite{K98} and elastodynamic systems \cite{AK1999}.
Moreover, numerical and mechanical experiments (\cite{BJCR2004}, \cite{Briffaut1999}, \cite{Urquiza2000}) have proved the efficiency of this feedback.

\bigskip
In this paper, after recalling the construction of  V. Komornik's feedback law and some results about the well-posedness of the open-loop problem (section \ref{Hypotheses}), we give a proof of two points that were not justified in \cite{K97}.
\begin{itemize}
\item[\textbullet]
The first  point (section \ref{WPCL}) is the well-posedness of the closed-loop problem with the explicit feedback introduced in \cite{K97}. Using the Riccati equation satisfied by $\Lom$, we introduce a ``dual'' closed-loop problem, which is easier to deal with because it does not involve the unbounded control operator $B$. Then we give a definition of \emph{the mild solution} of the initial closed-loop problem and we prove in Theroem \ref{thmWP} that this solution satisfies a variation of constants formula. To derive this formula, we adapt a representation formula of F. Flandoli \cite{Flan87} to the case of an algebraic Riccati equation. In Theorem \ref{thmWP2}, we prove that using a suitable extension $\wA$ of $A$, the  operator $\wA-BB^*\ILom$ is the infinitesimal generator of a strongly continuous group on the original state space.
\item[\textbullet]
The second point (section \ref{Decay}) consists in the justification of a formula, contained in Proposition \ref{RepIL}, that is used in \cite{K97} to prove the exponential decay of the solutions. We recall at the end of the paper how this formula is  used to obtain the exponential decay.
\end{itemize}

\emph{
The author thanks Professors I. Lasiecka and R. Triggiani for fruitful conversations on the subject of this work.
}
\section{A short review of the construction of the feedback and of the open-loop problem}\label{Hypotheses}
\subsection{Hypotheses and notations}
The state space $H$ and the control space $U$ are Hilbert spaces. We denote by $H'$ and $U'$ their duals
and by 
\begin{itemize}
\item[]
$J:U'\to U$  the canonical  isomorphism between $U'$ and $U$;
\item[]
$\wJ:H\to H'$  the canonical  isomorphism between $H$ and $H'$.
\end{itemize}

Moreover we make the following hypotheses :
\begin{itemize}
\item[\textbullet] {\bf(H1)} The operator $A : D(A)\subset H\to H$ is the infinitesimal generator of a strongly continuous group $e^{tA}$ on $H$. 
\footnote{
Thus his adjoint $A^*: D(A^*)\subset H'\to H'$ is also the infinitesimal generator of a strongly continous group $e^{tA^*}=(e^{tA})^*$ on $H'$.
}
\item[\textbullet] {\bf(H2)}
$B \in L(U,D(A^*)')$, where $D(A^*)'$ denotes the dual 
\footnote{Provided with the norm
\begin{equation*}
\|x\|^2_{D(A^*)}:=\|x\|^2_{H'}+\|A^*x\|^2_{H'},
\end{equation*}
$D(A^*)$ is a Hilbert space. Moreover,
\begin{equation*}
D(A^*)\subset H' \quad \Longrightarrow \quad H \subset D(A^*)'.
\end{equation*}
}
space of $D(A^*)$.
Identifying $D(A^*)''$ with $D(A^*)$, we denote by $B^*\in L(D(A^*),U')$ the adjoint of $B$. This implies the existence of a number $\lambda \in \C$ and a bounded operator $E\in L(U,H)$ such that
\begin{equation*}
B^*=E^*(A+\lambda I)^*.
\end{equation*}
\item[\textbullet] {\bf(H3)}
Given $T>0$, there exists a constant $c_1(T)>0$ such that
\begin{equation*}
\int_0^T\|B^*e^{-tA^*}x\|^2_{U'}dt\le c_1(T)\|x\|^2_{H'}
\end{equation*}
for all $x\in D(A^*)$. In the examples, this inequality represents a trace regularity result (see \cite{LT1983}). It is usually called the \emph{direct inequality}.
\item[\textbullet] {\bf(H4)}
There exists a number $T>0$ and a constant $c_2(T)>0$ such that
\begin{equation*}
c_2(T)\|x\|_{H'}^2\le \int_0^T\|B^*e^{-tA^*}x\|^2_{U'}dt
\end{equation*}
for all $x\in D(A^*)$. It is usually called the \emph{inverse} or \emph{observability inequality}.
\end{itemize}
\begin{rk}
Thanks to the assumptions (H1)-(H2), if the direct inequality in (H3) is satisfied for one $T>0$, then it is satisfied for all $T>0$. Moreover, the estimation remains true (up to a change of the constant in the right member) if we integrate on $(-T,T)$. Extending this inequality to all $x\in H'$ by density, the map $t\mapsto B^*e^{-tA^*}x$ can be seen as an element of $L^2_{\text{loc}}(\R;U').$
\end{rk}
\subsection{Construction of the feedback}\label{Construction}${}$

\bigskip
\emph {The operator $\Lom$.}
We suppose that the hypotheses (H1)-(H4) hold true (the number $T>0$ giving the observability inequality in (H4)) and we recall the construction of the feedback exposed in \cite{K97}, by defining a modified, weighted Gramian. We fix a number $\om>0$, set
\begin{equation*}
\Tom:=T+\frac{1}{2\om},
\end{equation*}
and we introduce a weight function  on the interval $[0,\Tom]$ :
\begin{equation*}
e_{\omega}(s):= 
\left\{
\begin{aligned}
&e^{-2\omega s} \quad \text{si} \quad  0\le s \le T \\
&2\omega e^{-2\omega T}(T_{\omega}-s) \quad \text{si } \quad  T \le s \le T_{\omega}. 
\end{aligned}
 \right.
\end{equation*}

Thanks to (H3) and (H4),
\begin{equation*}
\la\Lom x, y\ra_{H,H'}:=\int_0^{\Tom} \eom(s) \la JB^*e^{-sA^*}x, B^*e^{-sA^*}y \ra_{U,U'} ds
\end{equation*}
defines a positive-definite self-adjoint operator $\Lom \in L(H',H)$.
Hence $\Lom$ is invertible and we denote by $\ILom\in L(H,H')$ its inverse.

Actually the weight function $\eom$ has been chosen in such a way that the operator $\Lom$ is solution to an algebraic Riccati equation. We are going to derive this Riccati equation because it will play a key role in the analysis of the well-posedness of the closed-loop problem and the exponential decay of the solutions.

\bigskip
\emph{An algebraic Riccati equation.}
Let $x, y \in D((A^*)^2)$. We compute the integral 
\begin{equation}\label{integral}
\int_0^{\Tom}\frac{d}{ds}\Big [ \eom(s) \la JB^*e^{-sA^*}x, B^*e^{-sA^*y} \ra_{U,U'}\Big]  ds
\end{equation}
in two different ways. Note that the quantity between the brackets is differentiable in the variable $s$ thanks to the regularity of $x$ and $y$, and the hypothesis (H2) made on $B^*$.
\begin{itemize}
\item[\textbullet]
On the one hand, as $\eom(\Tom)=0$ and $\eom(0)=1$, the  above integral is 
\begin{equation*}
-\la JB^*x,B^*y\ra_{U,U'}.
\end{equation*}
\item[\textbullet]
On the other hand, by differentiating inside the integral, we obtain
\begin{gather*}
\int_0^{\Tom} \eom'(s) \la JB^*e^{-sA^*}x, B^*e^{-sA^*}y \ra_{U,U'} ds-\int_0^{\Tom} \eom(s) \la JB^*e^{-sA^*}A^*x, B^*e^{-sA^*}y \ra_{U,U'} ds \\
-\int_0^{\Tom} \eom(s) \la JB^*e^{-sA^*}x, B^*e^{-sA^*}A^*y \ra_{U,U'} ds.\\
\end{gather*}
\end{itemize}

The formula
\begin{equation*}
(Lx,y)_H:=-\int_0^{\Tom} \eom'(s)\la JB^*e^{-sA^*}\Lom^{-1}x, B^*e^{-sA^*}\Lom^{-1}y \ra_{U,U'} ds
\end{equation*}
defines a positive-definite self-adjoint operator $L\in L(H)$ because 
\begin{equation*}
\forall s\ge 0, \qquad -\eom'(s)\ge 2\omega \eom(s).
\end{equation*}
We set 
\begin{equation*}
C:=\sqrt{L}\in L(H).
\end{equation*}
For $x,y \in H$, we have
\begin{align*}
(Lx,y)_H&=(Cx,Cy)_H\\
&=\la Cx,\wJ Cy \ra_{H,H'}\\
&=\la x,C^*\wJ Cy\ra_{H,H'}
\end{align*}
where $C^*\in L(H')$ is the adjoint of $C$.
We can also remark the important
\footnote{
This estimation is important for the proof of the exponential decay of the solutions.
}
relation between $C$ and $\Lom^{-1}$ :
\begin{equation}\label{CLom}
C^*\wJ C \ge 2\omega \Lom^{-1}.
\end{equation}
Finally the second computation of the integral gives
\begin{align*}
&-(L\Lom x,\Lom y)_H-\la\Lom A^*x,y\ra_{H,H'}-\la\Lom x,A^*y\ra_{H,H'}\\
=&-\la C\Lom x,\widetilde{J}C\Lom y \ra_{H,H'}-\la\Lom A^*x,y\ra_{H,H'}-\la\Lom x,A^*y\ra_{H,H'}.
\end{align*}

Putting together the two computations, we obtain  the following algebraic Riccati equation satisfied by $\Lom$ :
\begin{multline}\label{ARE}
\la\Lom A^*x,y\ra_{H,H'}+\la\Lom x,A^*y\ra_{H,H'}\\
+\la C\Lom x,\widetilde{J}C\Lom y \ra_{H,H'}-\la JB^*x,B^*y\ra_{U,U'}=0,
\end{multline}
first for $x, y \in D((A^*)^2)$ and then for $x, y \in D(A^*)$ by density of $D((A^*)^2)$ in $D(A^*)$ for the norm $\|\cdot\|_{D(A^*)}$.

\bigskip
\emph{An integral form of the algebraic Riccati equation.}
We rewrite the Riccati equation \eqref{ARE} in an integral form, verified for $x, y \in H$ instead of $x, y \in D(A^*)$.
Set $x, y \in D(A^*).$ The equation \eqref{ARE} applied to $e^{-sA^*}x, e^{-sA^*}y \in D(A^*)$ gives
\begin{multline} \label{R1}
\la\Lom A^* e^{-sA^*}x,e^{-sA^*}y\ra_{H,H'}+\la\Lom e^{-sA^*}x,A^* e^{-sA^*}y\ra_{H,H'}\\+\la C\Lom e^{-sA^*}x,\widetilde{J}C\Lom e^{-sA^*}y\ra_{H,H'}-\la JB^*e^{-sA^*}x,B^* e^{-sA^*}y\ra_{U,U'}=0.
\end{multline}
Integrating \eqref{R1} bewteen $0$ and $t$ gives the following integral form of the Riccati equation \eqref{ARE} :
\begin{multline}\label{RI}
\la\Lom x,y\ra_{H,H'}
=\la \Lom e^{-tA^*}x,e^{-tA^*}y\ra_{H,H'}\\ 
-\int_0^t \la C\Lom e^{-sA^*}x,\widetilde{J}C\Lom e^{-sA^*}y \ra_{H,H'} ds
+\int_0^t \la JB^*e^{-sA^*}x,B^* e^{-sA^*}y\ra_{U,U'}ds.
\end{multline}
This relation remains true for $x, y \in H'$ by density of $D(A^*)$ in $H'$ for the \linebreak norm $\|\cdot\|_{H'}$.

\bigskip
\pagebreak
\emph{Rapid stabilization.}
Now let us recall the main result of \cite{K97}.
\begin{thm}[Komornik, {\cite[p. 1597]{K97}}]
Assume (H1)-(H4) for some $T>0$. Fix $\omega>0$ arbitrarily and set
\begin{equation*}
F:=-JB^*\ILom.
\end{equation*}
Then the operator $A+BF$ generates a strongly continuous group
\footnote{As it was already noted in \cite{K97}, we have to consider this affirmation in a weaker sense. More precisely, we will see that this is true if we replace $A$ by a suitable extension.}
 in H and the solutions of the closed-loop problem
\begin{equation*}
x'=Ax+BFx,\qquad x(0)=x_0
\end{equation*}
satisfy the estimate
\footnote{
$\|\cdot\|_{\om}$ defined by $\|x\|_{\om}^2:=\la \Lom^{-1}x,x\ra_{H',H}$, is a norm on $H$, equivalent to the usual norm thanks to the continuity and coercivity of $\ILom$.
}
\begin{equation*}
\|x(t)\|_{\om}\le \|x_0\|_{\om}e^{-\om t}
\end{equation*}
for all $x_0\in H$ and for all $t\ge 0$.
\end{thm}
\subsection{Well-posedness of the open-loop problem}\label{WPOL}
In this paragraph, we recall some results about the well-posedness of the open-loop problem
\begin{equation} \label{openloop}
\begin{cases}
x'(t)=Ax(t)+Bu(t), \qquad t \in \R, \\
x(0)=x_0,
\end{cases}
\end{equation}
where $u\in L^2_{\text{loc}}(\R;U)$. We would like to define a \emph{mild solution} of this problem that is continuous and takes its values in $H$. The dificulty comes from the fact that the control operator is unbounded and takes its values in the larger space $D(A^*)'$. The next proposition will give an answer.
\footnote{
This result is due to I. Lasiecka and R. Triggiani  who first proved it in the case of hyperbolic equations with Dirichlet boundary conditions (see \cite{LT1983}).
}
\begin{prop}[{\cite[p. 259-260]{BDP}}, {\cite[p. 648]{LT2} }] 
Fix $T>0$ and set 
\begin{equation*}
z(t):=\int_0^t e^{(t-s)A}Eu(s)ds,\qquad -T\le t\le T.
\end{equation*}
Then
\begin{itemize}
\item[\textbullet]
$z(t)\in D(A)$ for all $-T\le t\le T$;
\item[\textbullet]
$\|(A+\lambda I)z(t)\|_H\le k\|u\|_{L^2(-T,T;U)}$ for all $-T\le t\le T$, where $k>0$ is a constant independant of $u$;
\item[\textbullet]
$(A+\lambda I)z\in C([-T,T];H)$.
\end{itemize}
\end{prop}
\begin{deff} 
We define \emph{the mild solution} of \eqref{openloop} as the application
\begin{equation}\label{VC}
x(t)=e^{tA}x_0+(A+\lambda I)\int_0^te^{(t-s)A}Eu(s)ds
\end{equation}
which is continuous on $\R$ with values in $H$.
\end{deff}
\begin{rk} 
The relation \eqref{VC} is a variation-of-constants-type formula. If $B$ is bounded, this relation corresponds to 
\begin{equation*}
x(t)=e^{tA}x_0+\int_0^te^{(t-s)A}Bu(s)ds.
\end{equation*}
Moreover we can also write the relation \eqref{VC} by using the duality pairing :
\begin{equation*}
\la x(t),y\ra_{H,H'}=\la x_0,e^{tA^*}y\ra_{H,H'}+\int_0^t\la u(s),B^*e^{(t-s)A^*}y\ra_{U,U'}ds,
\end{equation*}
for all $y\in D(A^*)$.
\end{rk}

We end this section by recalling a regularity result. It concerns the solutions of the open-loop problem in the dual space $H'$
\begin{equation} \label{openloop2}
\begin{cases}
y'(t)=-A^*y(t)+g(t),\qquad t\in\R, \\
y(0)=y_0,
\end{cases}
\end{equation}
where $g \in L^1_{\text{loc}}(\R;H')$. This time, the source term does not involve any unbounded operator and the mild solution of \eqref{openloop2} is defined by the ``standard''  variation of constants formula (see \cite[p. 107]{Pazy}
\begin{equation}\label{VC2}
y(t)=e^{-tA^*}y_0+\int_0^t e^{-(t-r)A^*}g(r)dr,
\end{equation}
which is a continuous function from $\R$ to $H'$. Thanks to the  direct inequality stated in (H3), we can apply the operator $B^*$ to the solution of the homogeneous problem associated to \eqref{openloop2} (put $g=0$ in \eqref{openloop2}) and see this new function as an element of $L^2_{\text{loc}}(\R;U')$. Actually, this operation can  be generalized to the solutions of the inhomogeneous problem ($g$ can be $\neq 0$). We recall this result
\footnote{
This result was firstly stated in \cite{LT1983} in the case of hyperbolic equations with Dirichlet boundary conditions.
}
 in the
\begin{prop}[{\cite[pp. 92-93]{Flan87}} , {\cite[p. 648]{LT2}} ]\label{extreg}
 Fix $T>0$. There exists a constant $c>0$ such that for all $y_0 \in D(A^*)$ and all $g\in L^1(-T,T;D(A^*))$ we have the estimation 
\begin{equation*}
\int_{-T}^T \|B^*y(t)\|_{U'}^2 dt \le c\big(\|y_0\|_{H'}^2+\|g\|^2_{L^1(-T,T;U')}\big),
\end{equation*}
where $y$ is defined by \eqref{VC2}. By density, we can say that this estimation remains true for all initial data $y_0\in H'$ and all source terms $g\in L^1(-T,T;H')$.
\end{prop}
\section{Well-posedness of the closed-loop problem}\label{WPCL}
The aim of this section is to give a notion of solution to the closed-loop problem 
\begin{equation}\label{closedloop}
\begin{cases}
x'(t)=Ax(t)-BJB^*\ILom x(t),\qquad t \in \R, \\
x(0)=x_0.
\end{cases}
\end{equation}
As for the open-problem \eqref{openloop}, we can not use directly a variation of constants formula because of the unbounded perturbation ($-BJB^*\ILom$) of the infinitesimal generator \nolinebreak $A$.

Let us give the main idea for the  well-posedness of \eqref{closedloop}. 
The Riccati equation \eqref{ARE} can be rewritten \emph{formally} as
\begin{equation*}
A\Lom+\Lom A^*+\Lom C^*\wJ C\Lom -BJB^*=0.
\end{equation*}
By multiplying the above equation on both side by $\ILom$, we get
\begin{equation}\label{FARE2}
\ILom A+ A^*\ILom+C^*\wJ C -\ILom BJB^*\ILom=0.
\end{equation}
Now we multiply the operator $A-BJB^*\ILom$ on the left by $\ILom$ and on the right by $\Lom$ to get
\begin{equation*}
\ILom(A-BJB^*\ILom)\Lom=\ILom A \Lom -\ILom BJB^*=-A^*-C^*\wJ C\Lom,
\end{equation*}
the last equality being a consequence of \eqref{FARE2}.
\begin{rk}
The two operators $A-BJB^*$ and $-A^*-C^*\wJ C\Lom$ are (formally) \emph{conjugated} by the operator $\Lom$. 
\end{rk}
The advantage of working with the conjugated operator is that the perturbation ($-C^*\wJ C\Lom$)  is bounded. We are going to analyze the well-posedness of the closed-loop problem \eqref{closedloop} by using the solutions of the ``conjugated'' closed-loop problem
\begin{equation}\label{closedloop2}
\begin{cases}
y'(t)=-A^*y(t)-C^*\wJ C\Lom y(t), \qquad t \in \R, \\ 
y(0)=y_0,
\end{cases}
\end{equation}
whose well-posedness is already known.

The perturbation being bounded, the operator $-A^*-C^*\wJ C\Lom$, defined on $D(A^*)$ is the infinitesimal generator of a  strongly continuous group $V(t)$ on $H'$ (see \cite[p. 22 and p. 76]{Pazy}). Moreover, for all $t\in \R$ and all $y_0\in H'$ we have
\begin{equation}\label{Vmild}
V(t)y_0=e^{-tA^*}y_0-\int_0^t e^{-(t-r)A^*}C^*\wJ C\Lom V(r)y_0dr.
\end{equation}
\begin{deff}
Let $x_0 \in H$. We define \emph{the mild solution} of \eqref{closedloop} by
\begin{equation*}
U(t)x_0:=\Lom V(t) \ILom x_0. 
\end{equation*} 
\end{deff}
Now we prove that this notion of solution is ``coherent'' with the closed-loop problem \eqref{closedloop} in the sense that it satisfies a variation of constants formula, close to the one that we would formally use.
\begin{thm}\label{thmWP}
$U(t)$ is a strongly continuous group in $H$ whose generator is 
\begin{equation*}
A_U:=\Lom(-A^*-C^*\wJ C\Lom)\ILom; \qquad D(A_U)=\Lom D(A^*).
\end{equation*}
Moreover, it satisfies the variation of constants formula
\begin{equation}\label{repU}
\la U(t)x_0,y\ra=\la e^{tA}x_0,y\ra-\int_0^t\la JB^*\Lom^{-1}U(r)x_0,B^*e^{(t-r)A^*}y\ra dr,
\end{equation}
for all $x_0 \in H$ and $y\in H'.$
\end{thm}
\begin{rk}
The formula \eqref{repU} does not mean that $A-BJB^*\ILom$ is the infinitesimal generator of a group (or even a semigroup) but it justifies the choice of $U(t)$ to define the mild solution of the closed-loop problem \eqref{closedloop}. To justify that the integral in \eqref{repU} is meaningful, see the remark after the Lemma just below.
\end{rk}
\begin{rk}
Formula \eqref{repU} can be rewritten as
\begin{equation}\label{repU2}
U(t)x_0=e^{tA}x_0-(A+\lambda I)\int_0^t e^{(t-r)A}EJB^*\Lom^{-1}U(r)x_0dr,
\end{equation}
for all $x_0\in H$. We can show \eqref{repU2} first for $x_0\in D(A^*)$ and  extend it by density. 
\end{rk}

\pagebreak
The proof of Theorem \ref{thmWP} relies on the following representation formula of $\Lom$.
\begin{lem}\label{RepL1Lem}
Set $x,y \in H'$ and $t \in \R$. Then 
\begin{equation}\label{RepL1}
\la\Lom x,y\ra_{H,H'}=\la\Lom V(t)x,e^{-tA^*}y\ra_{H,H'}+\int_0^t\la JB^*V(s)x,B^*e^{-sA^*}y\ra_{U,U'} ds.
\end{equation}
\end{lem}
\begin{rk}
The integral in the above formula is meaningful. Indeed the first part of the bracket defines an element of $L^2_{\text{loc}}(\R;U)$ because of \eqref{Vmild} and the extended regularity result stated in Proposition \ref{extreg}. The second part of the bracket  defines an element of $L^2_{\text{loc}}(\R;U')$ thanks to the direct inequality stated in (H3).
\end{rk}
\begin{proof}[Proof of Theorem \ref{thmWP}]
At first, $U(t)$ is a $C_0$-group on $H$ because it is the conjugate group (by $\Lom$) of $V(t)$ . The relation between the infinitesimal generator of $V(t)$ and those of $U(t)$ is also a general fact about conjugate semigroups (see \cite[p. 43 and p. 59]{EN}).

To prove relation \eqref{repU}, we use relation \eqref{RepL1} in which we replace $x$ by $\ILom x_0$ and $y$ by $e^{tA^*}y$. Finally we use the definition of $U(t)$, that is 
$U(t)x_0=\Lom V(t) \ILom x_0.$
\end{proof}

\begin{proof}[Proof of Lemma \ref{RepL1Lem}]
F. Flandoli has proved in \cite{Flan87} a similar relation for the solution of a differential Riccati equation. We adapt his proof to the case of an algebraic Riccati equation.
The proof contains two steps : at first, we use the integral form of the Riccati equation \eqref{RI} and the variation of constants formula for $V$ \eqref{Vmild} to prove relation \eqref{RepL1}
modulo a rest. Then we show that this rest vanishes. 
\footnote{
In order to simplify the notations, we will omit the name of the spaces under the duality brackets in this proof.
}
\bigskip

{\bf First step.}
Fix $x,y \in H'$ and $t\in \R$. From \eqref{RI} and \eqref{Vmild} we have
\begin{align*}
&\la\Lom x,y\ra\\
=&\la \Lom \Big[e^{-tA^*}x\Big],e^{-tA^*}y\ra 
-\int_0^t \la C\Lom \Big[e^{-sA^*}x\Big],\wJ C\Lom e^{-sA^*}y \ra ds\\
&+\int_0^t \la JB^*\Big[e^{-sA^*}x\Big],B^* e^{-sA^*}y\ra ds\\
=&\la \Lom \Big[V(t)+\int_0^t e^{-(t-r)A^*}C^*\wJ C\Lom V(r)dr\Big]x,e^{-tA^*}y\ra\\
&-\int_0^t \la C\Lom \Big[V(s)+\int_0^se^{-(s-r)A^*}C^*\wJ C\Lom V(r)dr\Big]x,\wJ C\Lom e^{-sA^*}y \ra ds\\
&+\int_0^t \la JB^*\Big[V(s)+\int_0^se^{-(s-r)A^*}C^*\wJ C\Lom V(r)dr\Big]x,B^* e^{-sA^*}y\ra ds\\
=&\la \Lom V(t)x,e^{-tA^*}y\ra +\int_0^t \la JB^*V(s)x,B^* e^{-sA^*}y\ra ds+R.\\
\end{align*}

{\bf Second step.}
To obtain relation \eqref{RepL1}, we have to show that the rest $R$ vanishes. To lighten the writing, we set
\begin{equation*}
g(r):=C^*\wJ C \Lom V(r)x \in C(\R;H').
\end{equation*}
Let us rewrite the rest :
\begin{align*}
R=&\la \Lom \int_0^te^{-(t-r)A^*}g(r)dr,e^{-tA^*}y \ra \\
&-\int_0^t \la C\Lom V(s)x,\wJ C\Lom e^{-sA^*}y \ra ds\\
&-\int_0^t \la C\Lom \int_0^se^{-(s-r)A^*}g(r)dr,\wJ C\Lom e^{-sA^*}y \ra ds\\
&+\int_0^t \la JB^*\int_0^se^{-(s-r)A^*}g(r)dr,B^* e^{-sA^*}y\ra ds.\\
=:& R_1-R_2-R_3+R_4.
\end{align*}

\textbullet $\,$ We can also write $R_1$ as 
\begin{equation*}
R_1=\int_0^t \la \Lom e^{-(t-r)A^*}g(r),e^{-(t-r)A^*}e^{-rA^*}y \ra dr.
\end{equation*}
The integrand of the above integral  corresponds to the first term in the right member of \eqref{RI} by replacing $x$ by $C^*\wJ C\Lom V(r)x=g(r)$, $y$ by $e^{-rA^*}y$ and $t$ by $t-r$. Hence
\begin{align*}
R_1=&\int_0^t \la \Lom g(r),e^{-rA^*}y \ra dr\\
&+\int_0^t\Big[\int_0^{t-r}\la C\Lom e^{-sA^*}g(r),\wJ C\Lom e^{-sA^*}e^{-rA^*}y\ra ds\Big]dr\\
&-\int_0^t\Big[\int_0^{t-r}\la JB^*e^{-sA^*}g(r),B^*e^{-sA^*}e^{-rA^*}y\ra ds\Big]dr\\
=:&R_1'+R_2'-R_3'.
\end{align*}

\textbullet $\,$
We have
\begin{equation*}
R_1'=R_2.
\end{equation*}
The change of variable $\sigma:=s+r$ and Fubini's theorem  give 
\begin{align*}
R_2'=&\int_0^t\int_r^{t}\la C\Lom e^{-(\sigma-r)A^*}g(r),\wJ C\Lom e^{-\sigma A^*}y\ra d\sigma dr\\
=&\int_0^t\int_0^{\sigma}\la C\Lom e^{-(\sigma-r)A^*}g(r),\wJ C\Lom e^{-\sigma A^*}y\ra drd\sigma \\
=&R_3.
\end{align*}

\textbullet $\,$ Il remains to show that $R_3'=R_4$. Difficulties arise since the operator $B^*$ is unbounded. The idea is to construct two approximations 
$R_3'(n)$ and $R_4(n)$ for $R_3'$ and $R_4$. We show that $R_3'(n)=R_4(n)$ and that $R_3'(n)$ and $R_4(n)$ converge respectively to  $R_3'$ and $R_4$.

\begin{rk}
$A^*$ is the infinitesimal generator of a  $C_0$-group in $H'$. Hence for sufficiently large  $n\in \N$,  $n$ lies in the resolvant set of $A^*$. We set
\begin{equation*}
I_n:=n(nI-A^*)^{-1}\in L(H').
\end{equation*}
Then for all $x\in H'$, $I_n x \in D(A^*)$ and $I_n x \to x$ as $n\to \infty$ (see \cite[Lemma 3.2. p. 9]{Pazy}).
Moreover, the sequence $\|I_n\|$ is bounded from above independently of $n$. Indeed, as $A^*$ is the generator of a group, it results from Hille-Yosida theorem  (\cite[Theorem 6.3 p. 23]{Pazy}) that for sufficiently large $n\in \N$,
\begin{equation*}
\|I_n\|=\|n(nI-A^*)^{-1}\|\le \frac{n\alpha}{n-\beta},
\end{equation*}
where  $\alpha$ and $\beta$ are two positive constants.
\end{rk}

\textbullet $\,$ For $n$ sufficiently large, we set
\begin{equation*}
R_3'(n):=\int_0^t\Big[\int_0^{t-r}\la JB^*e^{-sA^*}I_ng(r),B^*e^{-sA^*}e^{-rA^*}y\ra ds\Big]dr.
\end{equation*}
The application between the duality bracket is  measurable on the product space $(0,t)\times(0,t)$.
\footnote{ 
The right side is measurable because it is the composition of two measurable functions. (we recall that $B^*e^{-tA^*}$ is well-defined in $L^2_{\text{loc}}(\R;U')$). In the left side  we can replace $B^*$ by $B^*_k:=E^*(A_k^*+\bar{\lambda} I)$ where $A^*_k\in L(H')$ is the \emph{Yosida approximation} of $A^*$ (see \cite{Pazy}). For all $x \in D(A^*)$, $B^*_k x\to B^*x$ as $k\to \infty$ and $B^*_k \in L(H',U')$. Hence, the left-hand side of the duality bracket is measurable as a simple limit of continuous (hence measurable) functions on $(0,t)\times(0,t)$.
}
Moreover
\begin{align*}
&\int_0^t\int_0^{t-r}\Big|\la JB^*e^{-sA^*}I_ng(r),B^*e^{-sA^*}e^{-rA^*}y\ra\Big| dsdr\\
\le& \int_0^t\int_0^t \Big|\la JB^*e^{-sA^*}I_ng(r),B^*e^{-sA^*}e^{-rA^*}y\ra\Big| dsdr\\
=& \int_0^t\Big[\int_0^t \Big|\la JB^*e^{-sA^*}I_ng(r),B^*e^{-sA^*}e^{-rA^*}y\ra\Big| ds\Big]dr \qquad \text{(Fubini-Tonelli)}\\
\le & c\int_0^t \|g(r)\|_{H'}\|e^{-rA^*}\|_{H'}dr\qquad \text{(Cauchy-Schwarz, direct inequality)}\\
<&\infty.
\end{align*}
Hence we can invert the order of the integrals in $R_3'(n)$. We get (first by doing the change of variable $\sigma:=s+r$) :
\begin{align*}
R_3'(n)=&\int_0^t\int_r^{t}\la JB^*e^{-(\sigma-r)A^*}I_ng(r),B^*e^{-\sigma A^*}y\ra d{\sigma}dr\\
=&\int_0^t\int_0^{{\sigma}}\la JB^*e^{-(\sigma-r)A^*}I_ng(r),B^*e^{-\sigma A^*}y\ra drd\sigma.
\end{align*}
Finally, $R_3'(n)=\int_0^t \varphi_n(r) dr$ et $R_3'=\int_0^t \varphi(r) dr$ with the evident notations. For all $0\le r\le t$, we have 
\begin{align*}
|\varphi_n(r)-\varphi(r)|=&\Big|\int_0^{t-r}\la JB^*e^{-sA^*}[I_ng(r)-g(r)],B^*e^{-sA^*}e^{-rA^*}y\ra ds\Big|\\
\le& \int_0^t\Big|\la JB^*e^{-sA^*}[I_ng(r)-g(r)],B^*e^{-sA^*}e^{-rA^*}y\ra \Big|ds\\
\le& c \|I_ng(r)-g(r)\|_{H'}\|e^{-rA^*}y\|_{H'} \quad \text{(Cauchy-Schwarz and direct inequality)}.
\end{align*}
Hence $\varphi_n(r)\to\varphi(r)$ as $n\to \infty$. Thanks to Cauchy-Schwarz, the direct inequality and because $\|I_n\|$ is bounded from above, we have
\begin{equation*}
|\varphi_n(r)|\le c\|I_n g(r)\|_{H'}\|e^{-rA^*}y\|_{H'}  \le c' \|g(r)\|_{H'}\|e^{-rA^*}y\|_{H'}.
\end{equation*}
We can apply the dominated convergence theorem : $R_3'(n)\to R_3$.

\textbullet $\,$ For sufficiently large $n$, we set
\begin{equation*}
R_4(n):=\int_0^t \la JB^*\int_0^se^{-(s-r)A^*}I_ng(r)dr,B^* e^{-sA^*}y\ra ds.
\end{equation*}
But $I_n$ et $e^{-(s-r)}A^*$ commute and 
\begin{align*}
B^*I_n&=E^*(A+\lambda I)^* n(nI-A^*)^{-1}\\
&=-nE^*+(n^2+n\lambda)E^*(nI-A^*)^{-1} \in L(H').
\end{align*}
Hence (see \cite[p. 139]{Bal1981} for interverting $B^*$ and the integral)
\begin{equation*}
B^*\int_0^se^{-(s-r)A^*}I_ng(r)dr=\int_0^sB^*I_ne^{-(s-r)A^*}g(r)dr 
\end{equation*}
and
\begin{equation*}
R_4(n)=\int_0^t\int_0^s \la JB^*I_ne^{-(s-r)A^*}g(r),B^* e^{-sA^*}y\ra drds=R_3'(n).
\end{equation*}
Finally, for all $0\le r\le t$, $I_n g(r) \to g(r)$ and $\|I_n g(r)\| \le c\|g(r)\|$, the right member being integrable on $(0,t)$. Thanks to the dominated convergence theorem, $I_n g \to g$ in $L^1(0,t;H')$. The estimation of proposition \ref{extreg} gives
\begin{equation*}
B^*\int_0^s e^{-(s-r)A^*} I_n g(r)dr \to B^* \int_0^s e^{-(s-r)A^*}g(r)dr
\end{equation*}
in $L^2(0,t;U')$. Hence $R_4(n)\to R_4$ and by unicity of the limit, $R_3'=R_4$.
\end{proof}

We can do a little better and link the infinitesimal generator of $U(t)$ to the original operators involved in the closed-loop problem \eqref{closedloop}.
\begin{thm}\label{thmWP2}
The operator $A: D(A)\subset H\to H$ can be extended to a bounded operator $\wA:H\to D(A^*)$. The operator 
\begin{equation*}
\wA-BJB^*\ILom
\end{equation*}
coincides with $A_U$ (the generator of $U(t)$) on $D(A_U)=\Lom D(A^*)$ i.e.
\begin{equation*}
\forall x_0 \in D(A_U),\qquad (\wA-BJB^*\ILom)x_0=A_U x_0 \in H.
\end{equation*}
\end{thm}
We first recall a classical extension result for the unbounded operator $A$ to a bounded operator on $H$ with values on the extrapolation space $D(A^*)'$ (see \cite[pp. 6-7]{LT1} and \cite[pp. 21-22]{CH1998}).
\begin{lem}\label{extension}
The operator $A:D(A)\subset H\to H$ admits a unique extension to an operator $\wA\in L(H,D(A^*)')$. Moreover this extension satisfies the relation
\begin{equation}\label{relextension}
\la \wA x,y\ra_{D(A^*)',D(A^*)}=\la x,A^*y\ra_{H,H'}.
\end{equation}
for all $x\in H$ and $y\in D(A^*)$.
\end{lem}
\begin{proof}[Proof of Lemma \ref{extension}] The unicity of such an extension is the consequence of the density of $D(A)$ in $H$.  Provided with the norm $\|\cdot\|_{D(A^*)}$, $D(A^*)$ is a Hilbert space and $A^*\in L(D(A^*),H')$. We denote by $\wA$ the (Banach-)adjoint of $A^*$ seen as a bounded operator between the Banach spaces $D(A^*)$ and $H$. Hence the definition of the (Banach-)adjoint gives
\begin{equation*}
\widetilde{A}\in L(H,D(A^*)')
\end{equation*}
and for all $x\in H$ and $y\in D(A^*)$,
\begin{equation*}
\la \widetilde{A}x,y\ra_{D(A^*)',D(A^*)}=\la x,A^*y\ra_{H,H'}
\end{equation*}
i.e. relation \eqref{relextension} is true. Moreover this new operator $\wA$ defines extension of $A$ i.e. the two operators coincides on $D(A)$.  Indeed from the above relation specialized to $x \in D(A)\subset H$, we get
\begin{equation*}
\forall y \in D(A^*), \quad \la \widetilde{A}x,y\ra_{D(A^*)',D(A^*)}=\la Ax,y\ra_{H,H'} \quad \Rightarrow \quad Ax=\wA x \in H. \qedhere
\end{equation*}
\end{proof}
\begin{proof}[Proof of Theorem \ref{thmWP2}]
Instead of returning to the Riccati equation \eqref{ARE}, we are going to differentiate the variation of constants formula \eqref{repU}.
We know that for $x_0 \in \Lom D(A^*)=D(A_U)$, the map 
\begin{equation*}
t\mapsto U(t)x_0
\end{equation*}
is differentiable and 
\begin{equation*}
\frac{d}{dt}U(t)x_0=A_U U(t)x_0.
\end{equation*}
In particular if $y\in H'$, then 
\footnote{
Again, when the name of spaces under the duality brackets are unnecessary, we omit them.
}
\begin{equation*}
\la\frac{d}{dt}U(t)x_0,y\ra=\la A_U U(t)x_0,y\ra.
\end{equation*}
Differentiating \eqref{repU} with respect to $t$, we want to link the generator $A_U$ and the operator $A-BJB^*\Lom^{-1}$ (a priori with values in $D(A^*)'$). We remark that defining the domain of the latter operator is not clear.
Let $x_0\in \Lom D(A^*)$ and $y\in D((A^*)^2)$. 

{\bf First step.} The map
\begin{equation*}
r\mapsto B^*\Lom^{-1}U(r)x_0
\end{equation*}
is continuous from $\R$ to $U'$. Indeed, setting $y_0:=\Lom^{-1}x_0\in D(A^*)$, we have
\begin{align*}
B^*\Lom^{-1}U(r)x_0&=B^*\Lom^{-1}\Lom V(r)\Lom^{-1}x_0\\
&=B^*V(r)y_0\\
&=E^*(A^*+\bar{\lambda}I)V(r)y_0\\
&=E^*(A^*+C^*\wJ C\Lom-C^*\wJ C\Lom+\bar{\lambda}I)V(r)y_0\\
&=-E^*(-A^*-C^*\wJ C\Lom)V(r)y_0+E^*(-C^*\wJ C\Lom+\bar{\lambda}I)V(r)y_0\\
&=-E^*V(r)(-A^*-C^*\wJ C\Lom)y_0+E^*(-C^*\wJ C\Lom+\bar{\lambda}I)V(r)y_0,\\
\end{align*}
the latter expression being continuous in $r$.
\begin{rk}
On $D(A^*)=D(-A^*-C^*\wJ C\Lom)$, the operators $V(r)$ et $-A^*-C^*\wJ C\Lom$ (generator of $V(r)$) commute (this is a general fact about semigroups) but a priori $V(r)$ and $A^*$ do not commute.
\end{rk}

{\bf Second step.}
The map
\begin{equation*}
s\mapsto B^*e^{sA^*}y
\end{equation*}
is differentiable on $\R$ with values in $U'$. 
Indeed, as $y\in D((A^*)^2)$, we have $(A^*+\bar{\lambda}I)y \in D(A^*)$ and
\begin{equation*}
B^*e^{sA^*}y=E^*e^{sA^*}(A^*+\bar{\lambda}I)y.
\end{equation*}
The latter expression is differentiable with respect to $s$ and its derivative is $B^*e^{sA^*}A^*y$.

{\bf Third step.} We deduce from the two previous steps that the map
\begin{equation*}
t\mapsto\int_0^t\la JB^*\Lom^{-1}U(r)x_0,B^*e^{(t-r)A^*}y\ra dr
\end{equation*}
is differentiable on $\R$ and its derivative is the map
\begin{equation*}
t\mapsto\int_0^t\la JB^*\Lom^{-1}U(r)x_0,B^*e^{(t-r)A^*}A^*y\ra dr +\la JB^*\Lom^{-1}U(t)x_0,B^*y\ra.
\end{equation*}
It results that given two (regular) data $x_0\in \Lom D(A^*)$ and $y\in D((A^*)^2)$, we can differentiate $\la U(t)x_0,y\ra$ with respect to $t$ and get
\begin{align*}
\frac{d}{dt}\la U(t)x_0,y\ra=&\la e^{At}x_0,A^*y\ra
-\int_0^t\la JB^*\Lom^{-1}U(r)x_0,B^*e^{(t-r)A^*}A^*y\ra dr\\
&-\la JB^*\Lom^{-1}U(t)x_0,B^*y\ra.
\end{align*}
Replacing $y$ by $A^*y$ in \eqref{repU} and reinjecting in the above relation, we obtain
\begin{equation}\label{diffrepU}
\frac{d}{dt}\la U(t)x_0,y\ra=\la U(t)x_0,A^*y\ra-\la JB^*\Lom^{-1}U(t)x_0,B^*y\ra.
\end{equation}
With the same regularity as above for $x_0$ et $y$, we have 
\begin{equation*}
\frac{d}{dt}\la U(t)x_0,y\ra_{H,H'}=\la A_UU(t)x_0,y\ra_{H,H'}=\la A_UU(t)x_0,y\ra_{D(A^*)',D(A^*)},
\end{equation*}
where  $A_U$ is the infinitesimal generator of $U(t)$. We recall from Lemma \ref{extension} that $A$ admits a unique extension to an aperator $\widetilde{A}\in L(H,D(A^*)')$. Thanks to this extension we can link $A_U$ and 
$A-BJB^*\Lom^{-1}$.  From  \eqref{diffrepU} and \eqref{relextension} we have, for 
$x_0\in \Lom D(A^*)$ and $y\in D((A^*)^2)$,
\begin{align*}
\frac{d}{dt}\la U(t)x_0,y\ra_{H,H'}&=\la U(t)x_0,A^*y\ra_{H,H'}-\la JB^*\Lom^{-1}U(t)x_0,B^*y\ra_{U,U'}\\
&=\la\widetilde{A}U(t)x_0,y\ra_{D(A^*)',D(A^*)}-\la BJB^*\Lom^{-1}U(t)x_0,y\ra_{D(A^*)',D(A^*)}\\
&=\la(\widetilde{A}-BJB^*\Lom^{-1})U(t)x_0,y\ra_{D(A^*)',D(A^*)}\\
&=\la A_UU(t)x_0,y\ra_{D(A^*)',D(A^*)}
\end{align*}
In particular, the latter equality is true for $t=0$.
Hence, given a fixed $x_0\in \Lom D(A^*)$, we have 
\begin{equation*}
\la(\widetilde{A}-BJB^*\Lom^{-1})x_0,y\ra_{D(A^*)',D(A^*)}=\la A_Ux_0,y\ra_{D(A^*)',D(A^*)},
\end{equation*}
for all $y\in D((A^*)^2)$. This relation remains true for all $y\in D(A^*)$ by density of $D((A^*)^2)$ in $D(A^*)$ (for the norm $\|\cdot\|_{D(A^*)}$).
Finally,
\begin{equation*}
\forall x_0\in \Lom D(A^*)=D(A_U),\qquad (\widetilde{A}-BJB^*\Lom^{-1})x_0=A_Ux_0 \in H. \qedhere
\end{equation*}
\end{proof}
\begin{rk}
With an unbounded control operator (i.e. $B\in L(U,D(A^*)')$), one can prove, through examples, that the domain of $A_U$ is not always included in the domain of $A$. Thus, in general, the extension $\wA$ is necessary in order to link $A_U$ and $A$ on $D(A_U)$ (i.e. we cannot omit the ``tilde'' in the above relation). This phenomenon does not appear with a bounded control operator (i.e. $B \in L(U,H)$): in that case, we can prove that the spaces $D(A)$ and $D(A_U)$ coincide.
\end{rk}
\section{Derivation of a representation formula for $\ILom$ \\and exponential decay}\label{Decay}
In this section, we give a justification to a representation formula for $\ILom$ involving the group $U(t)$. This corresponds to the formula (3.11) in \cite{K97}. We recall it as it is written in \cite{K97} : for all $s,t \in \R$,
\begin{gather*}
\ILom=U(t-s)^*\ILom U(t-s)\\
+\int_s^tU(\tau-s)^*(C^*\wJ C+\ILom BJB^*\ILom)U(\tau-s)d\tau.
\end{gather*}
This formula is used in \cite{K97} to prove the exponential decay of the solutions of the closed-loop system. Again, F. Flandoli derived an analog formula in the case of differential Riccati equations in \cite{Flan87}. We adapt his proof to the case of algebraic Riccati equations.

We first prove a similar representation formula for $\Lom$.
\begin{prop}\label{RepL2}  
For all $x,y \in H'$ and $t \in \R$
\begin{gather}
\la\Lom x,y\ra_{H,H'}=\la\Lom V(t)x,V(t)y\ra_{H,H'}\\ \nonumber
+\int_0^t \la JB^*V(s)x,B^*V(s)y\ra_{U,U'} ds
+\int_0^t\la C\Lom V(s)x,\widetilde{J}C\Lom V(s)y\ra_{H,H'} ds.
\end{gather}
\end{prop}
\begin{proof}[Proof]
It relies on the representation formula \eqref{RepL1} for $\Lom$ that we have already proved : for $x,y \in H'$,
\begin{equation*}
\la\Lom x,y\ra=\la\Lom V(t)x,[e^{-tA^*}y]\ra+\int_0^t\la JB^*V(s)x,B^*[e^{-sA^*}y]\ra ds.
\end{equation*}
In the right member of the above relation, we replace $e^{-tA^*}y$  and $e^{-sA^*}y$ by using the variation of constants formula \eqref{Vmild} for $V$ :
\begin{align*}
\la\Lom x,y\ra=&\la\Lom V(t)x,V(t)y\ra\\
&+\la\Lom V(t)x,\int_0^t e^{-(t-s)A^*}C^*\wJ C\Lom V(s)yds\ra\\
&+\int_0^t\la JB^*V(s)x,B^*V(s)y\ra ds\\
&+\int_0^t\la JB^*V(s)x,B^*\int_0^s e^{-(s-r)A^*}C^*\wJ C\Lom V(r)ydr\ra ds\\
=:&T_1+T_2+T_3+T_4.
\end{align*}
But
\begin{equation*}
T_2:=\int_0^t \la\Lom V(t-s)V(s)x,e^{-(t-s)A^*}C^*\wJ C\Lom V(s)y\ra ds.
\end{equation*}
Thanks to \eqref{RepL1}, applied to  $V(s)x$ instead of $x$, $C^*C\Lom V(s)y$ instead of y and $t-s$ instead of $t$, we have 
\begin{align*}
T_2=&\int_0^t \la\Lom V(s)x,C^*\wJ C\Lom V(s)y\ra ds\\
&-\int_0^t\int_0^{t-s}\la JB^*V(r)V(s)x,B^*e^{-rA^*}C^*\wJ C\Lom V(s)y\ra drds\\
=&\int_0^t \la C\Lom V(s)x,\wJ C\Lom V(s)y\ra ds\\
&-\int_0^t\int_0^{t-s}\la JB^*V(r+s)x,B^*e^{-rA^*}C^*\wJ C\Lom V(s)y\ra drds.
\end{align*}
The change of variable $\sigma:=r+s$ in the last term gives 
\begin{align*}
T_2=&\int_0^t \la C\Lom V(s)x,\wJ C\Lom V(s)y\ra ds\\
&-\int_0^t\int_s^t\la JB^*V(\sigma)x,B^*e^{-(\sigma-s)A^*}C^*\wJ C\Lom V(s)y\ra d\sigma ds\\
=&\int_0^t \la C\Lom V(s)x,\wJ C\Lom V(s)y\ra ds\\
&-\int_0^t\int_0^\sigma\la JB^*V(\sigma)x,B^*e^{-(\sigma-s)A^*}C^*\wJ C\Lom V(s)y\ra dsd\sigma.
\end{align*}
Hence we have  proved that
\begin{align*}
\la\Lom x,y\ra=&\la\Lom V(t)x,V(t)y\ra\\
&+\int_0^t\la JB^*V(s)x,B^*V(s)y\ra ds\\
&+\int_0^t \la C\Lom V(s)x,\wJ C\Lom V(s)y\ra ds\\
&+\int_0^t\la JB^*V(s)x,B^*\int_0^s e^{-(s-r)A^*}C^*\wJ C\Lom V(r)ydr\ra ds\\
&-\int_0^t\int_0^s\la JB^*V(s)x,B^*e^{-(s-r)A^*}C^*\wJ C\Lom V(r)y\ra drds. 
\end{align*}
We have already shown in the proof of Lemma \ref{RepL1Lem} that the two last terms in the above relation cancel each other. Hence the relation is proved. 
\end{proof}
\begin{prop}\label{RepIL}
For all $x, y \in H$ and $t \in \R$
\begin{gather}
\la\Lom^{-1} x,y\ra_{H',H}=\la\Lom^{-1} U(t)x,U(t)y\ra_{H',H}\\ \nonumber
+\int_0^t \la \wJ CU(s)x,CU(s)y\ra_{H',H} ds
+\int_0^t\la JB^*\Lom^{-1} U(s)x,B^*\Lom^{-1} U(s)y\ra_{U,U'} ds.
\end{gather}
\end{prop}
\begin{proof}[Proof]
We replace $x$ by $\Lom^{-1}x$ and $y$ by $\Lom^{-1}y$ in the relation given by the  \linebreak Proposition \ref{RepL2} :
\begin{align*}
\la x,\Lom^{-1}y\ra=&\la\Lom V(t)\Lom^{-1}x,V(t)\Lom^{-1}y\ra+\int_0^t \la JB^*V(s)\Lom^{-1}x,B^*V(s)\Lom^{-1}y\ra ds\\
&+\int_0^t\la\wJ  C\Lom V(s)\Lom^{-1}x, C\Lom V(s)\Lom^{-1}y\ra ds.
\end{align*}
Then, by definition of $U$,
\begin{gather*}
\la \Lom^{-1}x,y\ra=\la\Lom^{-1}U(t)x,U(t)y\ra\\
+\int_0^t \la JB^*\Lom^{-1}U(s)x,B^*\Lom^{-1}U(s)y\ra ds
+\int_0^t\la \wJ CU(s)x,CU(s)y\ra ds. \qedhere
\end{gather*}
\end{proof}
\begin{rk}
A simple change of variable implies that for all $s, t \in \R$ and all $x, y \in H$,
\begin{gather} \label{reldecay}
\la \Lom^{-1}x,y\ra=\la\Lom^{-1}U(t-s)x,U(t-s)y\ra\\ \nonumber
+\int_s^t \la JB^*\Lom^{-1}U(\tau-s)x,B^*\Lom^{-1}U(\tau-s)y\ra d\tau
+\int_s^t\la \wJ CU(\tau-s)x,CU(\tau-s)y\ra d\tau. 
\end{gather}
\end{rk}

Finally, let us recall the outline of the proof of the exponential decay of the solutions of the closed-loop problem \eqref{closedloop}. We denote by $x(t)$ the mild solution of \eqref{closedloop} i.e.
\begin{equation*}
x(t)=U(t)x_0.
\end{equation*}
Using the relation \eqref{reldecay} with $x=y=U(s)x_0=x(s)$, we have
\begin{gather*} 
\la\Lom^{-1} x(s),x(s)\ra=\la\Lom^{-1} x(t),x(t)\ra\\ 
+\int_s^t \la JB^*\Lom^{-1}x(\tau),B^*\Lom^{-1}x(\tau)\ra d\tau
+\int_s^t\la \wJ Cx(\tau),Cx(\tau)\ra d\tau. 
\end{gather*}
Let $0\le s\le t$.
The estimation \eqref{CLom} between $C$ and $\ILom$ and the positiveness of the second term of the right member in the above relation yield
\begin{equation*}
\la\Lom^{-1} x(s),x(s)\ra\ge\la\Lom^{-1} x(t),x(t)\ra
+2\omega \int_s^t\la \ILom x(\tau),x(\tau)\ra d\tau.
\end{equation*}
A Gronwall-type lemma (see \cite[p. 1599]{K97}) gives 
\begin{equation*}
\|x(t)\|_{\omega}^2\le \|x_0\|_{\omega}^2 e^{-2\omega t}\qquad \forall t\ge 0.
\end{equation*}



\begin{thebibliography}{10}

\bibitem{AK1999}
{\sc F.~Alabau and V.~Komornik}, {\em Boundary observability, controllability,
  and stabilization of linear elastodynamic systems}, SIAM J. Control Optim.,
  37 (1999), pp.~521--542 (electronic).

\bibitem{Bal1981}
{\sc A.~V. Balakrishnan}, {\em Applied Functional Analysis}, vol.~3 of
  Applications of Mathematics, Springer-Verlag, New York, second~ed., 1981.

\bibitem{BDP}
{\sc A.~Bensoussan, G.~Da~Prato, M.~C. Delfour, and S.~K. Mitter}, {\em
  Representation And Control of Infinite Dimensional Systems}, Systems \&
  Control: Foundations \& Applications, Birkh{\"a}user Boston Inc., Boston, MA,
  second~ed., 2007.

\bibitem{BJCR2004}
{\sc F.~Bourquin, M.~Joly, M.~Collet, and L.~Ratier}, {\em An efficient
  feedback control algorithm for beams: experimental investigations}, Journal
  of Sound and Vibration, 278 (2004), pp.~181--206.

\bibitem{Briffaut1999}
{\sc J.-S. Briffaut}, {\em M\'ethodes num\'eriques pour le contr\^ole et la
  stabilisation rapide des grandes strucutures flexibles}, PhD thesis, \'Ecole
  Nationale des Ponts et Chauss\'ees, 1999.

\bibitem{CH1998}
{\sc T.~Cazenave and A.~Haraux}, {\em An introduction to semilinear evolution
  equations}, vol.~13 of Oxford Lecture Series in Mathematics and its
  Applications, The Clarendon Press Oxford University Press, New York, 1998.

\bibitem{Coron07}
{\sc J.-M. Coron}, {\em Control and Nonlinearity}, vol.~136 of Mathematical
  Surveys and Monographs, American Mathematical Society, Providence, RI, 2007.

\bibitem{EN}
{\sc K.~Engel and R.~Nagel}, {\em One-Parameter Semigroups for Linear Evolution
  Equations}, Springer-Verlag, 1999.

\bibitem{Flan87}
{\sc F.~Flandoli}, {\em A new approach to the {L}-{Q}-{R} problem for
  hyperbolic dynamics with boundary control}, Lecture Notes in Control and
  Information Sciences, 102 (1987), pp.~89--111.

\bibitem{Kleinman1970}
{\sc D.~L. Kleinman}, {\em An easy way to stabilize a linear constant system},
  IEEE Transactions on Automatic Control,  (1970), p.~692.

\bibitem{K97}
{\sc V.~Komornik}, {\em Rapid boundary stabilization of linear distributed
  systems}, SIAM Journal on Control and Optimization, 35 (1997),
  pp.~1591--1613.

\bibitem{K98}
{\sc V.~Komornik}, {\em Rapid boundary stabilization of {M}axwell's equations},
  in \'{E}quations aux d{\'e}riv{\'e}es partielles et applications,
  Gauthier-Villars, {\'E}d. Sci. M{\'e}d. Elsevier, Paris, 1998, pp.~611--622.

\bibitem{KL05}
{\sc V.~Komornik and P.~Loreti}, {\em Fourier Series in Control Theory},
  Springer Monographs in Mathematics, Springer-Verlag, New York, 2005.

\bibitem{LT1983}
{\sc I.~Lasiecka and R.~Triggiani}, {\em Regularity of hyperbolic equations
  under {$L_{2}(0,\,T;L_{2}(\Gamma ))$}-{D}irichlet boundary terms}, Appl.
  Math. Optim., 10 (1983), pp.~275--286.

\bibitem{LT1}
{\sc I.~Lasiecka and R.~Triggiani}, {\em Control theory for partial
  differential equations: continuous and approximation theories. {I}}, vol.~74
  of Encyclopedia of Mathematics and its Applications, Cambridge University
  Press, Cambridge, 2000.
\newblock Abstract parabolic systems.

\bibitem{LT2}
\leavevmode\vrule height 2pt depth -1.6pt width 23pt, {\em Control theory for
  partial differential equations: continuous and approximation theories. {II}},
  vol.~75 of Encyclopedia of Mathematics and its Applications, Cambridge
  University Press, Cambridge, 2000.
\newblock Abstract hyperbolic-like systems over a finite time horizon.

\bibitem{Lions1988}
{\sc J.-L. Lions}, {\em Exact controllability, stabilizability and
  perturbations for distributed systems}, SIAM Rev.,  (1988), pp.~1--68.

\bibitem{Lukes1968}
{\sc D.~L. Lukes}, {\em Stabilizability and optimal control}, Funkcial. Ekvac.,
  11 (1968), pp.~39--50.

\bibitem{Pazy}
{\sc A.~Pazy}, {\em Semigroups of Linear Operators and Applications to Partial
  Differential Equations}, Springer-Verlag, 1992.

\bibitem{Russell1979}
{\sc D.~L. Russell}, {\em Mathematics of finite-dimensional control systems},
  vol.~43 of Lecture Notes in Pure and Applied Mathematics, Marcel Dekker Inc.,
  New York, 1979.
\newblock Theory and design.

\bibitem{Slemrod}
{\sc M.~Slemrod}, {\em A note on complete controllability and stabilizability
  for linear control systems in {H}ilbert space}, SIAM J. Control,  (1974),
  pp.~500--508.

\bibitem{Urquiza2000}
{\sc J.~M. Urquiza}, {\em Contr\^ole d'\'equations des ondes lin\'eaires et
  quasilin\'eaires}, PhD thesis, Universit\'e de Paris VI, 2000.

\end{thebibliography}
\end{document}